\documentclass{amsproc}
\usepackage{amssymb,hyperref,xy}
\xyoption{all}
\numberwithin{equation}{section}
\newtheorem{theorem}[equation]{Theorem}
\newtheorem{lemma}[equation]{Lemma}
\newtheorem{cor}[equation]{Corollary}
\theoremstyle{remark}
\newtheorem{remark}[equation]{Remark}
\newtheorem{example}[equation]{Example}
\newcommand{\C}{\mathbf C}
\newcommand{\Hilb}{\mathcal H}
\newcommand{\poset}{partially ordered set}
\newcommand{\Sp}{\mathrm{Sp}}

\newcommand{\Z}{\mathbf Z}
\theoremstyle{defition}
\newtheorem{defn}[equation]{Definition}

\DeclareMathOperator{\Aut}{Aut}
\DeclareMathOperator{\End}{End}

\title[Weil representations]{Combinatorics of finite abelian groups\\and Weil representations}
\author{Kunal Dutta}
\address{KD: Max-Planck-Institut f\"ur Informatik\\
Algorithms and Complexity\\
Campus E1 4, Room 319\\
66123 Saarbr\"ucken\\
Germany.}
\author{Amritanshu Prasad}
\address{AP: The Institute of Mathematical Sciences\\CIT campus, Taramani\\ Chennai 600113, India.}
\begin{document}
\begin{abstract}
  The Weil representation of the symplectic group associated to a finite abelian group of odd order is shown to have a multiplicity-free decomposition.
  When the abelian group is $p$-primary, the irreducible representations occurring in the Weil representation are parametrized by a partially ordered set which is independent of $p$.
  As $p$ varies, the dimension of the irreducible representation corresponding to each parameter is shown to be a polynomial in $p$ which is calculated explicitly.
  The commuting algebra of the Weil representation has a basis indexed by another \poset{} which is independent of $p$.
  The expansions of the projection operators onto the irreducible invariant subspaces in terms of this basis are calculated.
  The coefficients are again polynomials in $p$.
  These results remain valid in the more general setting of finitely generated torsion modules over a Dedekind domain.
\end{abstract}
\maketitle
\section{Introduction}
\subsection{Overview}
\label{sec:overview}
\footnote{2010 \emph{Mathematics Subject Classification.} Primary 11F27, 05E10. Secondary 81R05.}Heisenberg groups were introduced by Weyl \cite[Chapter~4]{Weyl50} in his mathematical formulation of quantum kinematics.
Best known among them are the Lie groups whose Lie algebras are spanned by position and momentum operators which satisfy Heisenberg's commutation relations.
Weyl also considered Heisenberg groups which are finite modulo centre,
such as the Pauli group (generated by the Pauli matrices), which he used to characterize the kinematics of electron spin.

A fundamental property of Heisenberg groups, predicted by Weyl and proved by Stone \cite{Stone30} and von Neumann \cite{vN31} for real Heisenberg groups is known as the Stone-von Neumann theorem.
Mackey \cite{Mackey49} extended this theorem to locally compact Heisenberg groups  (see Section~\ref{sec:basic-definitions} for the case that is pertinent to this paper, and \cite{SvN} for a more detailed and general exposition).
By considering Heisenberg groups associated to finite fields, local fields and ad\`eles, Weil \cite{MR0165033} demonstrated the importance of Heisenberg groups in number theory.

Weil exploited the Stone-von Neumann-Mackey theorem to construct a projective representation of a group of automorphisms of the Heisenberg group, now commonly known as the Weil representation.
Along with parabolic induction and the technique of Deligne and Lusztig \cite{MR0393266} using $l$-adic cohomology, the Weil representation is one of the most important techniques for constructing representations of reductive groups over finite fields (see G\'erardin \cite{Gerardin77} and Srinivasan \cite{MR528020}) or local fields (see G\'erardin \cite{MR0396859} and Moeglin-Vigneras-Waldspurger \cite{MR1041060}).

Tanaka \cite{MR0219635,MR0229737} showed how the Weil representation can be used to construct all the irreducible representations of $\mathrm{SL}_2(\Z/p^k\Z)$ for odd $p$ by looking at Weil representations associated to the abelian groups $\Z/p^k\Z\oplus \Z/p^l\Z$ for $l\leq k$.
However, most of the literature on Weil representations associated to finite abelian groups has focused on vector spaces over finite fields and on constructing representations of classical groups over finite fields.

The representation theory of groups over finite principal ideal local rings was initiated by Kloosterman \cite{MR0021032}, who studied $\mathrm{SL}_2(\Z/p^k\Z)$.
In contrast to general linear groups over finite fields, whose character theory was worked out by Green \cite{MR0072878}, the representation theory for general linear groups over these rings is quite hard.
It has been shown (Aubert-Onn-Prasad-Stasinski \cite{MR2607551} and Singla \cite{Poojareps}) that this problem is intricately related to the problem of understanding the representations of automorphism groups of finitely generated torsion modules over discrete valuation rings.
However, explicit constructions have been available either for a very small class of representations (Hill \cite{MR1334228,MR1311772}) or for a very small class of groups (Onn \cite{MR2456275}, Stasinski \cite{MR2588859}, Singla\cite{Poojareps}).

This article concerns the decomposition of the Weil representation of the full symplectic group associated to a finite abelian group of odd order (and more generally, a finite module of odd order over a Dedekind domain) into irreducible representations.
When the module in question is elementary (e.g., $(\Z/p\Z)^n$ for some odd prime $p$), it is well-known that the Weil representation, which may be realized on the space of functions on the abelian group, breaks up into two irreducible subspaces consisting of even and odd functions.
Besides this, only the case where all the invariant factors are equal (e.g., $(\Z/p^k\Z)^n$) has been understood completely (see Prasad \cite[Theorem~2]{MR1478492} for the case where $k$ is even, Cliff-McNeilly-Szechtman \cite{MR1783635} for the general case).
A small part of the decomposition has been explained in the general case (Cliff-McNeilly-Szechtman \cite{MR1971043}).
In their paper \cite{1101.0560}, Maktouf and Torasso have shown that
the restriction of the Weil representation of a symplectic group over a $p$-adic field to a maximal compact subgroup or to a maximal elliptic torus is multiplicity-free and have given an explicit description of the irreducible subrepresentations.

In this paper, we describe all the invariant subspaces for the Weil representation for all finite modules of odd order over a Dedekind domain.
To be specific, it is shown that the Weil representation has a multiplicity-free decomposition (Theorem~\ref{theorem:multiplicity-free}).
When the underlying finitely generated torsion module is primary of type $\lambda$ for some partition $\lambda$ (see (\ref{eq:11}) and (\ref{eq:13})), the irreducible components are parametrized by elements of a partially ordered set which depends only on $\lambda$, and not on the underlying ring.
As the local ring varies, for a fixed element of this \poset{}, the dimension of the corresponding representation is shown to be a polynomial in the order of its residue field whose coefficients do not depend on the ring (Theorem~\ref{theorem:main}).
These polynomials are computed explicitly (Theorem~\ref{theorem:dim}).
The centralizer algebra of the Weil representation also has a combinatorial basis indexed by a partially ordered set which depends only on $\lambda$ and not on the underlying ring.
The projection operators onto the irreducible invariant subspaces, when expressed in terms of this basis are also shown to have coefficients which are polynomials in the order of the residue field whose coefficients also do not depend on the ring (Theorem~\ref{theorem:alpha-qualitative}), and these polynomials are computed explicitly (Theorems~\ref{theorem:alpha-supp} and~\ref{theorem:alpha-exact}).
Thus the decomposition of the Weil representation into irreducible invariant subspaces is, despite its apparent complexity, combinatorial in nature.

The results in this paper could serve as a starting point from which more subtle constructions involving the Weil representation (such as Howe duality) which have worked so well in the case of classical groups over finite fields can be extended to groups of automorphisms of finitely generated torsion modules over a discrete valuation ring.

It is worth noting that every Heisenberg group that is finite modulo centre is isomorphic to one of the groups considered here (for the precise statement, see Prasad-Shapiro-Vemuri \cite{Prasad2010}, particularly, Section~3 and Corollary~5.7).
For example, the seemingly different Heisenberg groups used by Tanaka \cite{MR0219635} to construct representations in the principal series and cuspidal series of finite $\mathrm{SL}_2$ are isomorphic.
The difference lies in the realization of the special linear group as a group of automorphisms.
The decomposition of any Weil representation associated to a finite abelian group will therefore always be a refinement of one of the decompositions described in this paper.

In order to concentrate on the important ideas without being distracted by technicalities, the main body of this paper uses the setting of finite abelian groups. Section~\ref{sec:finite-modules} explains how to carry over the results to finitely generated torsion modules over discrete valuation rings and even more generally, finite modules over Dedekind domains.

To obtain our results, we use the combinatorial theory of orbits in finite abelian groups developed by us in \cite{MR2793603} (the relevant part is recalled in Section~\ref{sec:orbits}), well-known basic facts about Heisenberg groups and Weil representations which are recalled in Section~\ref{sec:basic-definitions} (of which simple proofs can be found in \cite{gdft}), and the standard combinatorial theory of \poset s, as set out in Chapter~3 of Stanley's book \cite{MR1442260}.

\subsection{Structure of the paper}
\label{sec:structure-paper}
In Section~\ref{sec:basic-definitions}, we recall the definition of the Heisenberg group and its Schr\"odinger representatution.
Following Weil \cite{MR0165033}, we deduce the existence of the Weil representation from the irreducibility and uniqueness of the Schr\"odinger representation.
Section~\ref{sec:problem} contains a precise formulation of our main problem - the decomposition of the Weil representation associated to a finite abelian group into irreducible summands.

In Section~\ref{sec:products}, we use the primary decomposition of finite abelian groups to reduce the main problem to the case of primary finite abelian groups.
In Section~\ref{sec:mult-orbits} we explain the relationship between the mutliplicities of the summands in the decomposition of the Weil representation and the number of orbits for the action of the symplectic group on the quotient of the Heisenberg group by its centre.

In Section~\ref{sec:orbits}, we recall the combinatorial theory of orbits and characteristic subgroups in a finite abelian group developed by the authors in \cite{MR2793603}.
An important order-reversing involution on the lattice of characteristic subgroups is introduced in Section~\ref{section:char-sub}.
The theory of \cite{MR2793603} is extended to symplectic orbits on the quotient of the Heisenberg group modulo its centre in Section~\ref{sec:symplectic-orbits}.

Section~\ref{sec:mult-one} contains the first major theorem of this article, namely that the decomposition of the Weil representation associated to a finite abelian group into simple representations is multiplicity-free (Theorem~\ref{theorem:multiplicity-free}).
This is achieved by computing the structure constants of its endomorphism algebra to show that this algebra is commutative (Lemma~\ref{lemma:product}).
It follows that the set of invariant subspaces of the Weil representation, partially ordered by inclusion, forms a Boolean lattice (Corollary~\ref{cor:Boolean-lattice}).

The task of describing the irreducible components of the Weil representation is carried out using combinatorial analysis in Sections~\ref{sec:construction-invariant-subs}--\ref{sec:parametrization}.
In Section~\ref{sec:construction-invariant-subs} two elementary types of invariant subspaces of the Weil representation are identified.
The first type are the subspaces of $L^2(A)$ consisting of even and odd functions; the second type are associated to so-called small order ideals
(these subspaces are far from being mutually disjoint and irreducible).
In Section~\ref{sec:components}, we describe a tensor product decomposition of the invariant subspaces associated to small order ideals.
In Section~\ref{sec:poset-invar-sub}, we refine the invariant subspaces of Section~\ref{sec:construction-invariant-subs} to construct a family of invariant subspaces, which as a poset under inclusion is described in terms of a combinatorially defined poset $Q_\lambda$.
In Section~\ref{sec:parametrization} the irreducible subrepresentations of the Weil representations are extracted from the invariant subspaces of Section~\ref{sec:poset-invar-sub} (Theorem~\ref{theorem:main}).
The rest of Section~\ref{sec:parametrization} is devoted to the explicit computation of the dimensions of these subrepresentations as well as formulae for the orthogonal projections onto them in terms of a natrual basis for the endomorphism algebra of the Weil representation.

Finally in Section~\ref{sec:finite-modules} we explain how to extend the ideas of this paper to analyze the Weil representation associated to any finite module over a Dedekind domain.
\subsection{Basic definitions}
\label{sec:basic-definitions}
Let $A$ be a finite abelian group of odd order.
Let $\hat A$ denote the Pontryagin dual of $A$.
This is the group of all homomorphisms $A\to U(1)$, where $U(1)$ denotes the group of unit complex numbers.
Let $K=A\times \hat A$.
For each $k=(x,\chi)\in K$, the unitary operator on $L^2(A)$ defined by
\begin{equation*}
  W_k f(u)=\chi(u-x/2)f(u-x) \text{ for all } f\in L^2(A),\; u\in A
\end{equation*}
is called a Weyl operator.
These operators satisfy
\begin{equation*}
  W_kW_l=c(k,l)W_{k+l} \text{ for all } k,l\in K,
\end{equation*}
where, if $k=(x,\chi)$ and $l=(y,\lambda)$, then
\begin{equation*}
  c(k,l)=\chi(y/2)\lambda(x/2)^{-1}.
\end{equation*}
Observe that $c(k,l)$ is bimultiplicative, for example, $c(k,l+l') = c(k,l)c(k,l')$ for all $k,l,l'\in K$.

The subgroup 
\begin{equation*}
  H=\{c W_k|c\in U(1),k\in K\}
\end{equation*}
of the group of unitary operators on $L^2(A)$ is called the Heisenberg group associated to $A$.
This group is known to physicists as a generalized Pauli group or a Weyl-Heisenberg group.
As defined here, it comes with a unitary representation on $L^2(A)$, called the Schr\"odinger representation.
Mackey's generalization \cite[Theorem~1]{Mackey49} of the Stone-von Neumann theorem applies:
\begin{theorem}
  \label{theorem:SvN}
  The Schr\"odinger representation of $H$ is irreducible.
  Let $\rho:H\to U(\Hilb)$ (here $U(\Hilb)$ is the group of unitary operators on a Hilbert space $\Hilb$) be an irreducible unitary representation of $H$ such that $\rho(c W_0)=c\mathrm{Id}_{\Hilb}$ for every $c\in U(1)$.
  Then there exists, up to scaling, a unique isometry $W:L^2(A)\to \Hilb$ such that
  \begin{equation*}
    WW_k=\rho(W_k)W \text{ for all } k\in K.
  \end{equation*}
\end{theorem}
If $g$ is an automorphism of $K$ such that
\begin{equation}
  \label{eq:2}
  c(g k,g l)=c(k,l) \text{ for all } k,l\in K
\end{equation}
then $\rho_g:H\to U(L^2(A))$ defined by
\begin{equation*}
  \rho_g(c W_k)=c W_{g(k)} \text{ for all } c\in U(1),\; k\in K
\end{equation*}
is an irreducible unitary representation of $H$ on $L^2(A)$ such that $\rho(c W_0)=c\mathrm{Id}_{L^2(A)}$.
By Theorem~\ref{theorem:SvN}, there exists a unitary operator $W_g$ on $L^2(A)$ such that $W_g W_k=W_{g(k)}W_g$ for all $k\in K$.
Writing $W_g^*$ for the adjoint of the unitary operator $W_g$, we have:
\begin{equation}
  \label{eq:1}
  W_g W_kW_g^*=W_{g(k)} \text{ for all } k\in K.
\end{equation}
If $g_1$ and $g_2$ are two such automorphisms, both $W_{g_1g_2}$ and $W_{g_1}W_{g_2}$ intertwine the Schr\"odinger representation with $\rho_{g_1g_2}$, and hence must differ by a unitary scalar:
\begin{equation*}
  W_{g_1}W_{g_2}=c(g_1,g_2)W_{g_1g_2} \text{ for some } c(g_1,g_2)\in U(1).
\end{equation*}
Let $\Sp(K)$ be the group of all automorphisms $g$ of $K$ which satisfy (\ref{eq:2}). We have shown that $g\mapsto W_g$ is a projective representation of $\Sp(K)$ on $L^2(A)$. This representation is known as the Weil representation.
\begin{remark}
  The operators $W_g$, for $g\in \Sp(K)$ can be normalized in such a way that $c(g_1,g_2)=1$ for all $g_1,g_2$ (see Remark~\ref{remark:ordinary-rep}).
  Thus the Weil representation can be taken to be an ordinary representation of $\Sp(K)$.
\end{remark}
\begin{remark}
  The overlap of notation between the Weyl operators and the Weil representation is suggested by (\ref{eq:1}), which implies that they can be combined to construct a representation of $H\rtimes \Sp(K)$.
  The operators in this representation are precisely the unitary operators which normalize $H$.
  The resulting group is sometimes known as a Clifford group or a Jacobi group.
  It plays a prominent role in the stabilizer formalism for quantum error-correcting codes (see Chapter X of Nielsen and Chuang \cite{Nielsen-Chuang}).
\end{remark}
\subsection{Formulation of the problem}
\label{sec:problem}
We investigate the decomposition
\begin{equation}
  \label{eq:12}
  L^2(A)=\bigoplus_{\pi\in \widehat{\Sp(K)}} m_\pi \Hilb_\pi
\end{equation}
into irreducible representations.
Here $\widehat{\Sp(K)}$ denotes the set of equivalence classes of irreducible unitary representations of $\Sp(K)$ and, for each $\pi:\Sp(K)\to U(\Hilb_\pi)$ in $\widehat{\Sp(K)}$, $m_\pi$ denotes the multiplicity of $\pi$ in the Weil representation.
Although the Weil representation is defined only up to multiplication by a scalar representation, the multiplicities and dimensions of the irreducible representations occurring in the decomposition are invariant under such twists (see Remark~\ref{remark:projective-equiv}).
As explained in Section~\ref{sec:overview}, the outcome of this paper is an understanding of this decomposition.
\section{Product decompositions}
\label{sec:products}
We shall recall and apply a well-known observation on Weil representations associated to a product of abelian groups (see \cite[Corollary~2.5]{Gerardin77}).
\subsection{Projective equivalence}
Since Weil representations are defined only up to scalar factors, we use a definition of equivalence of representations that is weaker than unitary equivalence:
\begin{defn}
  [Projective equivalence]
  \label{defn:proj-equiv}
  Let $G$ be a group and $\rho_i:G\to U(\Hilb_i)$ for $i=1,2$ be two unitary representations of $G$.
  We say that $\rho_1$ and $\rho_2$ are projectively equivalent if there exists a homomorphism $\chi:G\to U(1)$ such that $\rho_2$ is unitarily equivalent to $\rho_1\otimes \chi$.
\end{defn}
\begin{remark}
  \label{remark:projective-equiv}
  If, as a representation of $G$,
  \begin{equation*}
    \Hilb_i = \bigoplus_{\pi \in \hat G} m^{(i)}_\pi \Hilb_\pi
  \end{equation*}
  is the decomposition of $\Hilb_i$ into irreducibles for representations as in Definition~\ref{defn:proj-equiv}, then $m^{(2)}_{\pi\otimes\chi}=m^{(1)}_\pi$, so there is a bijection between the sets of irreducible representations of $G$ that appear in $\Hilb_1$ and $\Hilb_2$ which preserves multiplicities and dimensions.
\end{remark}
\subsection{Tensor product decomposition}
If $A$ admits a product decomposition $A=A'\times A^{\prime\prime}$, then
\begin{equation}
  \label{eq:9}
  L^2(A)=L^2(A')\otimes L^2(A^{\prime\prime}).
\end{equation}
Let $K'=A'\times \widehat{A'}$, $K^{\prime\prime}=A^{\prime\prime}\times \widehat{A^{\prime\prime}}$.
Thus $K=K'\times K''$.
Let $S'$ and $S^{\prime\prime}$ be subgroups of $\Sp(K')$ and $\Sp(K^{\prime\prime})$ respectively.
Then $S=S'\times S''$ is a subgroup of $\Sp(K)$.
\begin{theorem}
  \label{theorem:product-Weil}
  The Weil representation of $S$ on $L^2(A)$ is projectively equivalent to the tensor product of the Weil representation of $S'$ on $L^2(A')$ and the Weil representation of $S''$ on $L^2(A'')$.
\end{theorem}
\begin{proof}
  By (\ref{eq:1}), the Weil representations of $S'$ and $S''$ satisfy
  \begin{equation*}
    W_{g'}W_{k'}W_{g'}^*=W_{g'(k')}\text{ and } W_{g''}W_{k''}W_{g''}^*=W_{g''(k'')}
  \end{equation*}
  for all $g'\in S'$, $g''\in S''$, $k'\in K'$ and $k''\in K''$, whence
  \begin{equation*}
    (W_{g'}\otimes W_{g''})(W_{k'}\otimes W_{k''})(W_{g'}\otimes W_{g''})^*=W_{g'(k')}\otimes W_{g''(k'')}.
  \end{equation*}
  Since $W_{k'}\otimes W_{k''}$ coincides with $W_{(k',k'')}$ under the isomorphism (\ref{eq:9}), $W_{g'}\otimes W_{g''}$ satisfies the defining identity (\ref{eq:1}) for the Weil representation of $S$ on $L^2(A)$.
\end{proof}
\subsection{Primary decomposition}
\label{sec:prim-decomp}
Every finite abelian group has primary decomposition
\begin{equation*}
  A = \prod_{p \text{ prime}} A_p,
\end{equation*}
where $A_p$ is the subgroup of elements of $A$ annihilated by some power of $p$.
Writing $K_p$ for $A_p\times \widehat{A_p}$,
\begin{equation*}
  K=\prod_p K_p\quad \text{ and } \quad \Sp(K)=\prod_p \Sp(K_p).
\end{equation*}
Theorem~\ref{theorem:product-Weil}, when applied to the primary decomposition gives
\begin{cor}
  \label{cor:primary-decomposition-Weil}
  The Weil representation of $\Sp(K)$ on $L^2(A)$ is projectively equivalent to the tensor product over those primes $p$ for which $A_p\neq 0$ of the Weil representations of $\Sp(K_p)$ on $L^2(A_p)$.
\end{cor}
  In view of Corollary~\ref{cor:primary-decomposition-Weil}, it suffices to consider the case where $A$ is a finite abelian $p$-group for some odd prime $p$.

\section{Multiplicities and orbits}
\label{sec:mult-orbits}
We now recall the relation between the decomposition of the Weil representation and orbits in $K$ \cite{gdft}.
\subsection{An orthonormal basis}
\begin{lemma}
  \label{lemma:onb}
  The set $\{W_k|k\in K\}$ of Weyl operators is an orthonormal basis of $\End_\C L^2(A)$.
\end{lemma}
\begin{proof}
  For each $k\in K$ and $T\in \End_\C L^2(A)$, let
  \begin{equation*}
    \tau(k)T=W_k T W_k^*.
  \end{equation*}
  Then $k\mapsto \tau(k)$ is a unitary representation of $K$ on $\End_\C L^2(A)$.
  If $k=(x,\chi)$ and $l=(y,\lambda)$ are two elements of $K$, then
  \begin{eqnarray*}
    \tau(k)W_l&=&W_k W_l W_k^*\\
    &=&W_k W_l (W_l W_k)^* W_l\\
    &=&c(k,l) W_{k+l} c(l,k)^{-1} W_{l+k}^* W_l\\
    &=&\chi(y)\lambda(x)^{-1}W_l.
  \end{eqnarray*}
  Thus the $W_l$'s are eigenvectors for the action of $K$ with distinct eigencharacters.
  Therefore they form an orthonormal set of operators.
  Since $|K|=|A|^2=\dim\End_\C L^2(A)$, this orthonormal set is a basis.
\end{proof}
\subsection{Endomorphisms}
By Lemma~\ref{lemma:onb}, every $T\in \End_\C L^2(A)$ has a unique expansion
\begin{equation}
  \label{eq:3}
  T=\sum_{k\in K} T_k W_k, \text{ with each } T_k\in \C.
\end{equation}
\begin{theorem}
  \label{theorem:orbit-mult}
  For every subgroup $S$ of $\Sp(K)$,
  \begin{equation*}
    \End_S L^2(A)=\{T\in \End_\C L^2(A)|T_k=T_{g(k)} \text{ for all } g\in S,\; k\in K\}.
  \end{equation*}
\end{theorem}
\begin{proof}
  Note that $T\in \End_S L^2(A)$ if and only if $W_g T W_g^*=T$ for all $g\in S$.
  Expanding $T$ as in (\ref{eq:3}) and using the defining identity (\ref{eq:1}) for $W_g$ gives the theorem.
\end{proof}
Now suppose that as a representation of $S$, $L^2(A)$ has the decomposition
\begin{equation*}
  L^2(A)=\bigoplus_{\pi\in \hat S}m_{\pi,S}\Hilb_\pi.
\end{equation*}
Then, together with Schur's lemma, Theorem~\ref{theorem:orbit-mult} implies
\begin{cor}
  \label{cor:orbit-mult}
  If $S\backslash K$ denotes the set of $S$-orbits in $K$,
  \begin{equation*}
    \sum_{\pi\in \hat S} m_{\pi,S}^2 =|S\backslash K|.
  \end{equation*}
\end{cor}

\section{Orbits and characteristic subgroups}
\label{sec:counting-orbits}
We first recall the theory of orbits (under the full automorphism
group) and characteristic subgroups in a finite abelian group from \cite{MR2793603}.
We then see how it applies to $\Sp(K)$-orbits in $K$.

\subsection{Orbits}
\label{sec:orbits}
Every finite abelian $p$-group is isomorphic to
\begin{equation}
  \label{eq:11}
  A=\Z/p^{\lambda_1}\Z\times \cdots \times \Z/p^{\lambda_l}\Z
\end{equation}
for a unique sequence $\lambda=(\lambda_1\geq\cdots\geq\lambda_l)$ of positive integers (in other words, a partition).
Henceforth, we assume that $A$ is of the above form.
For each partition $\lambda$, let
\begin{equation*}
  P_\lambda = \big\{(v,k)|k\in \{\lambda_1,\ldots,\lambda_l\},\; 0\leq v<k\big\}.
\end{equation*}
Say that $(v,k)\geq (v',k')$ if and only if $v'\geq v$ and $k'-v'\leq k-v$.
This relation is a partial order on $P_\lambda$.
For $x\in \Z/p^k\Z$, let
\begin{equation*}
  v(x)=\max\{0\leq v\leq k|x\in p^v\Z/p^k\Z\}.
\end{equation*}

For $a=(a_1,\ldots,a_l)\in A$, let $I(a)$ be the order ideal in $P_\lambda$ generated by $(v(a_i),\lambda_i)$ with $a_i\neq 0$ in $\Z/p^{\lambda_i}\Z$.
\begin{example}
\begin{figure}[h]
  \begin{equation*}
    \begin{array}{ccc}
      \tiny{
      \begin{xy}
        (0,0)*{(4,5)};
        (5,5)*{(3,4)}**@{-};
        (0,10)*{(3,5)}**@{-};
        (5,15)*{(2,4)}**@{-};
        (0,20)*{(2,5)}**@{-};
        (5,25)*{(1,4)}**@{-};
        (0,30)*{(1,5)}**@{-};
        (5,35)*{(0,4)}**@{-};
        (0,40)*{(0,5)}**@{-};
        (5,5)*{(3,4)};
        (20,20)*{(0,1)}**@{-};
        (5,35)*{(0,4)}**@{-};
      \end{xy}
      }
      &\quad\quad\quad &
      \begin{xy}
        (0,0)*{\bullet};
        (5,5)*{\bullet}**@{-};
        (0,10)*{\bullet}**@{-};
        (5,15)*{\bullet}**@{-};
        (0,20)*{\circ}**@{-};
        (5,25)*{\circ}**@{-};
        (0,30)*{\circ}**@{-};
        (5,35)*{\circ}**@{-};
        (0,40)*{\circ}**@{-};
        (5,5)*{\circ};
        (20,20)*{\circ}**@{-};
        (5,35)*{\circ}**@{-};
      \end{xy}\\
      \text{The poset } P_{(5,4,4,1)} & & \text{The order ideal } I(p^4,p^2,p^3,0)
    \end{array}
  \end{equation*}
  \caption{}
  \label{fig:1}
\end{figure}
  \label{eg:Plambda}
  When $\lambda=(5,4,4,1)$, and $a=(p^4,p^2,p^3,0)$ the Hasse diagram of $P_\lambda$ is shown on the left hand side of Figure~\ref{fig:1}.
  The ideal $I(a)$ is represented by black dots on the right hand side of Figure~\ref{fig:1}.
Note that the elements of $P_\lambda$ are arranged in such a way that $k$ is constant along verticals and decreases from left to right.
\end{example}
\begin{theorem}
  \cite[Theorem~4.1]{MR2793603}
  \label{theorem:degeneration}
  For $a,b\in A$, $b$ is the image of $a$ under an endomorphism of $A$ if and only if $I(b)\subset I(a)$.
\end{theorem}
Given $x=(v,k)\in P_\lambda$, let $e(x)$ denote the element in $A$ all of whose entries are zero except for the left-most entry with $\lambda_i=k$, which is $p^v$.
For an order ideal $I$ in $P_\lambda$ denote by $\max I$ the set of maximal elements in $I$ and let
\begin{equation*}
  a(I)=\sum_{x\in \max I} e(x).
\end{equation*}
Let $G$ denote the group of all automorphisms of $A$.
\begin{theorem}
  \cite[Theorem~5.4]{MR2793603}
  \label{theorem:orbits}
  The map $I\mapsto a(I)$ gives rise to a bijection from the set of order ideals in $P_\lambda$ to the set of $G$-orbits in $A$.
\end{theorem}
The elements $a(I)$, as $I$ varies over the order ideals in $P_\lambda$, can be taken as representatives of the orbits.
The inverse of the function of Theorem~\ref{theorem:orbits} is given by $a\mapsto I(a)$.
\subsection{Characteristic subgroups}
\label{section:char-sub}
For an order ideal $I\subset P_\lambda$
\begin{equation*}
  A_I=\{a\in A|I(a)\subset I\}
\end{equation*}
is a characteristic subgroup of $A$ of order $p^{[I]}$, where $[I]$ denotes the number of elements in $I$, counted with multiplicity (the multiplicity of $(v,k)$ is the number of times that $k$ occurs in the partition $\lambda$, see \cite[Theorem~7.3]{MR2793603}).
Every characteristic subgroup of $A$ is of the form $A_I$ for some order ideal $I\subset P_\lambda$.
In fact, $I\mapsto A_I$ is an isomorphism of the lattice of order ideals in $P_\lambda$ onto the lattice of characteristic subgroups of $A$.
Thus, the lattice of characteristic subgroups of $A$ is a finite distributive lattice \cite[Section~3.4]{MR1442260}.
If $B$ is any group isomorphic to $A$, and $\phi:A\to B$ is an isomorphism, then since $A_I$ is characteristic, the image $B_I=\phi(A_I)$ does not depend on the choice of $\phi$.
Consequently, it makes sense to talk of the subgroup $\hat A_I$ of $\hat A$, which is the image of $A_I$ under any isomorphism $A\to \hat A$.

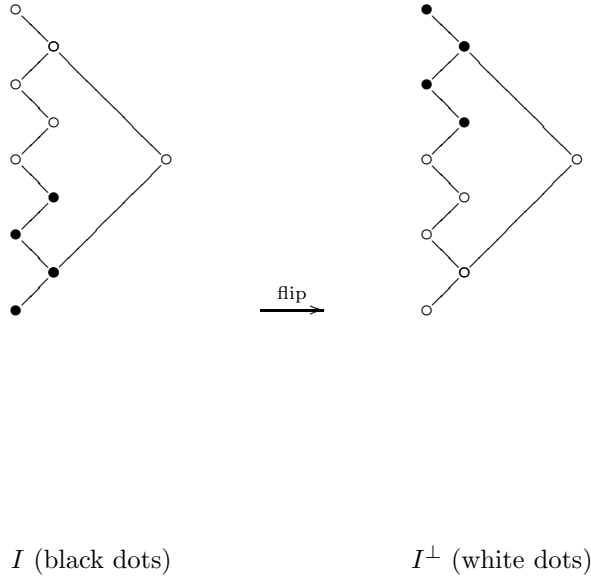
\begin{figure}[h]
  \begin{equation*}
  \xymatrix{
      \begin{xy}
    (0,0)*{\bullet};
    (5,5)*{\bullet}**@{-};
    (0,10)*{\bullet}**@{-};
    (5,15)*{\bullet}**@{-};
    (0,20)*{\circ}**@{-};
    (5,25)*{\circ}**@{-};
    (0,30)*{\circ}**@{-};
    (5,35)*{\circ}**@{-};
    (0,40)*{\circ}**@{-};
    (5,5)*{\circ};
    (20,20)*{\circ}**@{-};
    (5,35)*{\circ}**@{-};
  \end{xy}
  &\ar[r]^{\text{flip}}&&
\begin{xy}
    (0,0)*{\circ};
    (5,5)*{\circ}**@{-};
    (0,10)*{\circ}**@{-};
    (5,15)*{\circ}**@{-};
    (0,20)*{\circ}**@{-};
    (5,25)*{\bullet}**@{-};
    (0,30)*{\bullet}**@{-};
    (5,35)*{\bullet}**@{-};
    (0,40)*{\bullet}**@{-};
    (5,5)*{\circ};
    (20,20)*{\circ}**@{-};
    (5,35)*{\circ}**@{-};
  \end{xy}\\
  I \text{ (black dots)} &&& I^\perp \text{ (white dots)}
  }
\end{equation*}
\caption{The involution on order ideals.}
\label{fig:2}
\end{figure}
For each order ideal $I\subset P_\lambda$, its annihilator
\begin{equation*}
  A_I^\perp:=\{\chi\in \hat A: \chi(a)=1 \text{ for all } a\in A_I\}
\end{equation*}
is a characteristic subgroup of $\hat A$.
Therefore, there exists an order ideal $I^\perp\subset P_\lambda$ such that $A_I^\perp=\hat A_{I^\perp}$.
Clearly, $I\mapsto I^\perp$ is an order reversing involution of the set of order ideals in $P_\lambda$.
The Hasse diagram of $P_\lambda$ has a horizontal axis of symmetry.
$I^\perp$ can be visualized as the complement of the reflection of $I$ about this axis of $I$ (see Figure~\ref{fig:2}).
\subsection{Symplectic orbits}
\label{sec:symplectic-orbits}
\begin{theorem}
  \label{theorem:symplectic-orbits}
  The map $I\mapsto (a(I),0)$ (here $0$ denotes the identity element of $\hat A$) gives rise to a bijection from the set of order ideals in $P_\lambda$ to the set of $\Sp(K)$-orbits in $K$.
\end{theorem}
\begin{proof}
  We first show that each $\Sp(K)$-orbit in $K$ intersects $A\times\{0\}$.
  Let $e_1,\ldots,e_l$ denote the generators of $A$, so $e_i$ is the element whose $i$th coordinate is $1$ and all other coordinates are $0$.
  Each element $a\in A$ has an expansion
  \begin{equation}
    \label{eq:4}
    a=a_1e_1+\cdots+a_le_l \text{ with } 0\leq a_i< p^{\lambda_i} \text{ for each } i\in \{1,\ldots,l\}
  \end{equation}
  Let $\epsilon_j$ denote the unique element of $\hat A$ for which
  \begin{equation*}
    \epsilon_j(e_k)=e^{2\pi i \delta_{j k} p^{-\lambda_j}}. 
  \end{equation*}
  Then each element $\alpha\in \hat A$ has an expansion
  \begin{equation}
    \label{eq:5}
    \alpha=\alpha_1\epsilon_1+\cdots+\alpha_l\epsilon_l \text{ with } 0\leq \alpha_i< p^{\lambda_i} \text{ for each } i\in \{1,\ldots,l\}.
  \end{equation}
  Let $k=(a,\alpha)\in K$, with $a$ and $\alpha$ as in (\ref{eq:4}) and (\ref{eq:5}) respectively.
  The automorphism of $K$ which takes $e_i\mapsto \epsilon_i$ and $\epsilon_i\mapsto -e_i$ while preserving all other generators $e_j$ and $\epsilon_j$ with $j\neq i$, lies in $\Sp(K)$.
  In terms of coordinates, it has the effect of interchanging $a_i$ and $\alpha_i$ up to sign.
  Using this automorphism, we may arrange that $v(a_i)\leq v(\alpha_i)$ for each $i$.
  Therefore, there exists an integer $b_i$ such that $b_i a_i\equiv \alpha_i\mod p^{\lambda_i}$.
  Let $B_i:A\to \hat A$ be the homomorphism which takes $e_i$ to $b_i\epsilon_i$ and all other generators $e_j$ with $j\neq i$ to $0$.
  Then the automorphism of $K$ which takes $(a,\alpha)$ to $(a,\alpha-B_i(a))$ also lies in $\Sp(K)$.
  This has the effect of changing $\alpha_i$ to $0$.
  Repeating this process for each $i$ allows us to reduce $(a,\alpha)$ to $(a,0)$ as claimed.

  Now, for every automorphism $g$ of $A$, the automorphism $(a,\alpha)\mapsto (g(a),\hat g^{-1}(\alpha))$ lies in $\Sp(K)$ (here $\hat g$ is the automorphism of $\hat A$ defined by $\hat g(\chi)(a)=\chi(g(a))$ for $a\in A$ and $\chi \in \hat A$).
  Such automorphisms can be used to reduce $(a,0)$ further to an element of the form $(a(I),0)$ for some order ideal $I\subset P_\lambda$.
  Since, for distinct $I$'s, these elements are in distinct $\Aut(K)$-orbits, they must also be in distinct $\Sp(K)$-orbits.
\end{proof}

\section{Multiplicity one}
\label{sec:mult-one}

\subsection{Relation to commutativity}
\label{sec:relat-comm}
Suppose that the decomposition of the Weil representation onto irreducible representations is given by
\begin{equation}
  \label{eq:6}
  L^2(A)=\bigoplus_{\pi \in \widehat{\Sp(K)}} m_\pi \Hilb_\pi.
\end{equation}
Then by Schur's lemma, the endomorphism algebra of $L^2(A)$ is a sum of matrix algebras:
\begin{equation*}
  \End_{\Sp(K)}L^2(A) = \bigoplus_{\pi\in \widehat{\Sp(K)}} M_{m_{\pi}\times m_\pi}(\mathbf C).
\end{equation*}
It follows that $m_\pi\leq 1$ for every $\pi\in \widehat{\Sp(K)}$ if and only if the ring $\End_{\Sp(K)}L^2(A)$ of endomorphisms of the Weil representations is commutative.
For each order ideal $I\subset P_\lambda$ let $O_I$ denote the $\Sp(K)$-orbit of $(a(I),0)$ in $K$ and let
\begin{equation*}
  T_I=\sum_{k\in O_I} W_k.
\end{equation*}
By Theorems~\ref{theorem:orbit-mult} and~\ref{theorem:symplectic-orbits}, the set of all $T_I$ as $I$ varies over the order ideals $I\subset P_\lambda$ is a basis of $\End_{\Sp(K)}L^2(A)$.

Let $K_I=A_I\times \hat A_I$ ($\hat A_I$ as in Section~\ref{section:char-sub}) and define
\begin{equation}
  \label{eq:8}
  \Delta_I=\sum_{k\in K_I} W_k.
\end{equation}
Since
\begin{equation*}
  K_I = \coprod_{J\subset I} O_J,
\end{equation*}
we have
\begin{equation*}
  \Delta_I=\sum_{J\subset I} T_J.
\end{equation*}
Thus, the elements $\Delta_I$ are obtained from the basis elements $T_I$ of $\End_{\Sp(K)}L^2(A)$ by a unipotent upper-triangular transformation.
Hence, 
\begin{equation*}
  \{\Delta_I \mid I\in J(P_\lambda)\}
\end{equation*}
is also a basis of $\End_{\Sp(K)}L^2(A)$.
Thus if $\Delta_I$ commutes with $\Delta_J$ for all $I,J\in J(P_\lambda)$, then $\End_{\Sp(K)}L^2(A)$ is a commutative algebra.

Therefore, in order to show that $m_\pi\leq 1$ for each $\pi\in \widehat{\Sp(K)}$, it suffices to show that for any two order ideals $I,J\subset P_\lambda$, $\Delta_I$ and $\Delta_J$ commute.
This will follow from the calculation in Section~\ref{sec:calculation}.

\subsection{Calculation of the product}
\label{sec:calculation}
\begin{lemma}
  \label{lemma:product}
  For any two order ideals $I,J\subset P_\lambda$,
  \begin{equation*}
    \Delta_I\Delta_J=|K_{I\cap J}|\Delta_{(I\cap J)^\perp \cap (I\cup J)}.
  \end{equation*}
\end{lemma}
\begin{proof}
  The coefficient of $W_k$ in $\Delta_I\Delta_J$ is
  \begin{equation}
    \label{eq:7}
    \sum_{x\in K_I,y\in K_J, x+y=k} c(x,y).
  \end{equation}
  From the definition of $I(a)$, it is easy to see that $I(a+b)\subset I(a)\cup I(b)$.
  Therefore, $A_I+A_J\subset A_{I\cup J}$ and hence $K_I+K_J\subset K_{I\cup J}$.
  It follows that the sum (\ref{eq:7}) is $0$ unless $k\in K_{I\cup J}$.
  Suppose $x_0\in K_I$ and $y_0\in K_J$ are such that $x_0+y_0=k$.
  Then the sum (\ref{eq:7}) becomes
  \begin{eqnarray*}
    \sum_{l\in K_I\cap K_J} c(x_0+l,y_0-l) & = & c(x_0,y_0)\sum_{l\in K_I\cap K_J} c(l,y_0)c(l,x_0)\\
    & = & c(x_0,y_0) \sum_{l\in K_I\cap K_J} c(l,k)\\
    & = &
    \begin{cases}
      c(x_0,y_0)|K_I\cap K_J| &\text{ if } k\in (K_I\cap K_J)^\perp,\\
      0 &\text{ otherwise}.
    \end{cases}
  \end{eqnarray*}
  Observe that $K_I\cap K_J=K_{I\cap J}$.
  It remains to show that, for every $k\in K_{I\cup J}$, there exist $x_0\in K_I$ and $y_0\in K_J$ such that $k=x_0+y_0$ and $c(x_0,y_0)=1$.
  Since $\Delta_I\Delta_J$ is constant on $\Sp(K)$-orbits in $K$, we may use Theorem~\ref{theorem:symplectic-orbits} to assume that $k=(a(I'),0)$ for some order ideal $I'\subset I\cup J$.
  We have $\max I'\subset I\cup J$.
  Let $I'_1$ be the order ideal generated by $(\max I')\cap I$, and $I'_2$ be the order ideal generated by $(\max I')-I$.
  Then $a(I')=a(I'_1)+a(I'_2)$.
  Clearly $(a(I'_1),0)$ and $(a(I'_2),0)$ have the properties required of $x_0$ and $y_0$.
\end{proof}

\subsection{Multiplicity one}
We have proved
\begin{theorem}
  \label{theorem:multiplicity-free}
  In the decomposition (\ref{eq:6}) of the Weil representation of $\Sp(K)$, $m_\pi\leq 1$ for every isomorphism class $\pi$ of irreducible representations of $\Sp(K)$.
\end{theorem}
Every $\Sp(K)$-invariant subspace is completely determined by the subset of $\widehat{\Sp(K)}$ consisting of representations that occur in it.
Therefore
\begin{cor}
  \label{cor:Boolean-lattice}
  The set of $\Sp(K)$-invariant subspaces of $L^2(A)$, partially ordered by inclusion, forms a finite Boolean lattice.
\end{cor}

\section{Elementary invariant subspaces}
\label{sec:construction-invariant-subs}
In this section, we construct some elementary invariant subspaces for the Weil representation of $\Sp(K)$ on $L^2(A)$.
In Section~\ref{sec:poset-invar-sub}, we will use these subspaces and the results of Section~\ref{sec:components} to construct enough invariant subspaces to carve out all the irreducible subspaces.

\subsection{Small order ideals}
\begin{defn}
  [Small order ideal]
  \label{defn:small-ideal}
  An order ideal $I\subset P_\lambda$ is said to be small if $I\subset I^\perp$, with $I^\perp$ as in Section~\ref{section:char-sub}.
\end{defn}
For example, the order ideal $I$ in Figure~\ref{fig:2} is small.
\subsection{Interpreting some $\Delta_I$'s}
\begin{lemma}
  \label{lemma:Delta-I}
  For each order ideal $I\subset P_\lambda$, let $\Delta_I$ be as in (\ref{eq:8}).
  \begin{enumerate}
  \item \label{item:2}$\Delta_{P_\lambda}f(u)=|A|f(-u)$ for all $f\in L^2(A)$ and $u\in A$.
  \item \label{item:1}For every small order ideal $I\subset P_\lambda$, $|A_I|^{-2}\Delta_I$ is the orthogonal projection onto the subspace of $L^2(A)$ consisting of functions supported on $A_{I^\perp}$ and invariant under translations in $A_I$.
  \end{enumerate}
\end{lemma}
\begin{proof}
  For any order ideal $I\subset P_\lambda$, we have
  \begin{equation*}
    \Delta_I f(u)  = \sum_{x\in A_I} \sum_{\chi \in \hat A_I} \chi(u-x/2) f(u-x).
  \end{equation*}
  The inner sum is $f(u-x)$ times the sum of values of a character of $\hat A_I$, which vanishes if this character is non-trivial, namely if $u-x/2\notin A_{I^\perp}$, and is $|A_I|$ otherwise.
  Therefore,
  \begin{eqnarray}
    \nonumber
    \Delta_If(u) & = & |A_I|\sum_{x\in A_I\cap (2u+A_{I^\perp})} f(u-x)\\
    \label{eq:17}
    & = & |A_I|\sum_{x\in (u+A_I)\cap (-u+A_{I^\perp})} f(x).
  \end{eqnarray}
  Taking $I=P_\lambda$ in (\ref{eq:17}) gives \ref{item:2}.

  Now suppose that $I\subset I^\perp$.
  If $u\notin A_{I^\perp}$ then $(u+A_I)\cap (-u+A_{I^\perp})=\emptyset$, so that $\Delta_If(u)=0$.
  If $u\in A_{I^\perp}$, then the sum (\ref{eq:17}) is over $u+A_I$, so $|A_I|^{-2}\Delta_I$ is the averaging over $A_I$-cosets
  from which \ref{item:1} follows.
\end{proof}

\subsection{Even and odd functions}
\label{sec:even-odd-functions}
\begin{theorem}
  \label{theorem:even-odd}
  The subspaces of $L^2(A)$ consisting of even and odd functions are invariant under $\Sp(K)$.
\end{theorem}
\begin{proof}
  By \ref{item:2},
  \begin{equation}
    \label{eq:22}
    [(\mathrm{id}_{L^2(A)}\pm |A|^{-1}\Delta_{P_\lambda})/2] f(u)=(f(u)\pm f(-u))/2.
  \end{equation}
  These operators are the orthogonal projections onto the subspaces of even and odd functions in $L^2(A)$.
  Since these operators commute with $\Sp(K)$ (by Theorem~\ref{theorem:orbit-mult}), their images are $\Sp(K)$-invariant subspaces of $L^2(A)$.
\end{proof}
\begin{remark}
  [The Weil representation is an ordinary representation]
  \label{remark:ordinary-rep}
  The subspaces of even and odd functions on $A$ have dimensions $(|A|+1)/2$ and $(|A|-1)/2$ respectively.
  For each $g\in \Sp(K)$, let $W_g^+$ and $W_g^-$ denote the restrictions of $W_g$ to these spaces.
  Taking the determinants of the identities
  \begin{equation*}
    W_{g_1}^\pm W_{g_2}^\pm=c(g_1,g_2)W_{g_1 g_2}^\pm
  \end{equation*}
  gives the identities
  \begin{gather*}
    \det W_{g_1}^+\det W_{g_2}^+=c(g_1,g_2)^{(|A|+1)/2}\det W_{g_1g_2}^+\\
    \det W_{g_1}^-\det W_{g_2}^-=c(g_1,g_2)^{(|A|-1)/2}\det W_{g_1g_2}^-.
  \end{gather*}
  Dividing the first equation by the second and rearranging gives:
  \begin{equation*}
    c(g_1,g_2)=\frac{\alpha(g_1)\alpha(g_2)}{\alpha(g_1g_2)}
  \end{equation*}
  for all $g_1,g_2\in \Sp(K)$, when $\alpha:G\to U(1)$ is defined by
  \begin{equation*}
    \alpha(g)=\det(W_g^+)/\det(W_g^-) \text{ for all } g\in \Sp(K).
  \end{equation*}
  Therefore, if each $W_g$ is replaced by $\alpha(g)^{-1}W_g$, then $g\mapsto W_g$ is a representation of $\Sp(K)$ on $L^2(A)$.
  This argument seems to be well known. 
  It has appeared before in Adler-Ramanan \cite[Appendix I]{MR1621185}, and again in Cliff-McNeilly-Szechtman \cite{MR1783635}.
\end{remark}
\subsection{Invariant spaces corresponding to small order ideals}
\label{sec:invar-spac-corr}
Since $\Delta_I$ commutes with $\Sp(K)$, its image is an $\Sp(K)$-invariant subspace of $L^2(A)$.
An immediate consequence of \ref{item:1} is the following theorem:
\begin{theorem}
  \label{theorem:characteristic-invariant}
  For each small order ideal $I\subset P_\lambda$, the subspace of $L^2(A)$ consisting of functions supported on $A_{I^\perp}$ which are invariant under translations in $A_I$ is an $\Sp(K)$-invariant subspace of $L^2(A)$.
\end{theorem}
\begin{remark}
  [Alternative description]
  For $f\in L^2(A)$ recall that its Fourier transform is the function on $\hat A$ defined by
  \begin{equation*}
    \hat f(\chi)=\sum_{a\in A} f(a)\overline{\chi(a)} \text{ for each } \chi\in \hat A.
  \end{equation*}
  For any subgroup $B$ of $A$, the Fourier transforms of functions invariant under translations in $B$ are the functions supported on the annihilator subgroup $B^\perp$ of $A$ (consisting of characters which vanish on $B$), and Fourier transforms of functions supported on $B$ are the functions which are invariant under $B^\perp$.
  Therefore, the functions supported on $A_{I^\perp}$ which are invariant under $A_I$ are precisely the functions supported on $A_{I^\perp}$ whose Fourier transforms are supported on $\hat A_{I^\perp}$.
  They are also the functions invariant under translations in $A_I$ whose Fourier transforms are invariant under translations in $\hat A_I$.
  \label{remark:alternative}
\end{remark}
Identify $L^2(A_{I^\perp}/A_I)$ with the space of functions in $L^2(A)$ which are supported on $A_{I^\perp}$ and invariant under translations in $A_I$.
Let $K(I)=A_{I^\perp}/A_I\times \widehat{A_{I^\perp}/A_I}$.
$K(I)$ can be identified with $(A_I{^\perp}\times \hat A_{I^\perp})/(A_I\times \hat A_I)$.
Thus $K(I)$ is a quotient of one characteristic subgroup of $K$ by another.
Therefore the action $\Sp(K)$ on $K$ descends to an action on $K(I)$ giving rise to a homomorphism $\Sp(K)\to \Sp(K(I))$.
The defining condition (\ref{eq:1}) for the Weil representation ensures:
\begin{theorem}
  \label{theorem:invar-I}
  For every small order ideal $I\subset P_\lambda$, the Weil representation of $\Sp(K)$ on $L^2(A_{I^\perp}/A_I)$ is projectively equivalent to the representation obtained by composing the Weil representation of $\Sp(K(I))$ on $L^2(A_{I^\perp}/A_I)$ with the homomorphism $\Sp(K)\to \Sp(K(I))$.
\end{theorem}
\section{Component decomposition}
\label{sec:components}

\subsection{Connected components of a \poset}
A \poset{} is said to be connected if its Hasse diagram is a connected graph.
A connected component of a \poset{} is a maximal connected induced subposet.
Every \poset{} can be written as the disjoint union of its connected components in the sense of \cite[Section~3.2]{MR1442260}.
Denote the set of connected components of a poset $P$ by $\pi_0(P)$.

\subsection{Connected components of $J-I$}
\label{sec:conn-comp-j-i}
Suppose that $I\subset J$ are two order ideals in $P_\lambda$.
Each connected component $C\in \pi_0(J-I)$ determines a segment (namely, a contiguous set of integers) $S_C$ in $\{1,\ldots,l\}$:
\begin{equation*}
  S_C=\{1\leq k\leq l|(v,k)\in C\text{ for some } v\}.
\end{equation*}
The $S_C$'s are pairwise disjoint, but their union may be strictly smaller than $\{1,\ldots,l\}$.
Write $S_0$ for the complement of $\coprod_{C\in \pi_0(I^\perp-I)}S_C$ in $\{1,\ldots,l\}$.
It will be convenient to write $\tilde \pi_0(I^\perp-I)=\pi_0(I^\perp-I)\coprod \{0\}$.
Define partitions $\lambda(C)=(\lambda_k|k\in S_C)$ for each $C\in \tilde \pi_0(I^\perp-I)$.
Then $P_{\lambda(C)}$ is the induced subposet of $P_\lambda$ consisting of those pairs $(v,k)\in P_\lambda$ for which $k\in S_C$.
Let $I(C)$ and $J(C)$ be the ideals in $P_{\lambda(C)}$ obtained by intersecting $I$ and $J$ respectively with $P_{\lambda(C)}$.

For example, if $\lambda=(5,4,4,1)$ and $I$ is the order ideal in the diagram on the left in Figure~\ref{fig:2} and $J=I^\perp$, then $I^\perp-I$ is depicted in the diagram on the left in Figure~\ref{fig:3}.
As the diagram on the right shows, the induced subposet $I^\perp-I$ has two connected components, $C_1$ and $C_2$, with $\lambda(C_1)=(5)$ and $\lambda(C_2)=(1)$. Moreover, $\lambda(0)=(4,4)$.
\begin{figure}
  \begin{equation*}
    \begin{array}{ccc}
      \begin{xy}
        (0,0)*{\circ};
        (5,5)*{\circ}**@{-};
        (0,10)*{\circ}**@{-};
        (5,15)*{\circ}**@{-};
        (0,20)*{\bullet}**@{-};
        (5,25)*{\circ}**@{-};
        (0,30)*{\circ}**@{-};
        (5,35)*{\circ}**@{-};
        (0,40)*{\circ}**@{-};
        (5,5)*{\circ};
        (20,20)*{\bullet}**@{-};
        (5,35)*{\circ}**@{-};
      \end{xy}
      &\quad\quad\quad
      &
      \begin{xy}
        (0,0)*{};
        (0,20)*{\bullet};
        (0,40)*{};
        (20,20)*{\bullet};
      \end{xy}\\
      I^\perp-I\text{ inside } P_{5,4,4,1} & & I^\perp-I\text{ by itself}
    \end{array}
  \end{equation*}
  \caption{}
  \label{fig:3}
\end{figure}
\begin{lemma}
  \label{lemma:ideals-components}
  Let $I\subset J$ be two order ideals in $P_\lambda$.
  For each $C\in \pi_0(J-I)$ let $L(C)$ be an order ideal in $C$.
  Let
  \begin{equation*}
    L=I\coprod \Big(\coprod_{C\in \pi_0(J-I)} L(C)\Big).
  \end{equation*}
  Then $L$ is an order ideal in $P_\lambda$.
\end{lemma}
\begin{proof}
  Since $\coprod_{C\in \pi_0(J-I)} L(C)$ is an order ideal in $J-I$, its union with $I$ is an order ideal in $P_\lambda$.
\end{proof}
\begin{cor}
  \label{cor:ideal-components}
  If $I\subset J$ are two order ideals in $P_\lambda$ and $C$ and $D$ are distinct components of $J-I$,
  then the intersection with $P_{\lambda(C)}$ of the order ideal in $P_\lambda$ generated by $J(D)$ is contained in $I(C)$.
\end{cor}
\begin{proof}
  By Lemma \ref{lemma:ideals-components}, $J(D)\cup I$ is an order ideal in $P_\lambda$.
  Therefore, it contains the order ideal in $P_\lambda$ generated by $J(D)$.
  If $C$ and $D$ are distinct connected components of $J-I$, then $(J(D)\cup I)\cap P_\lambda(C)=I(C)$.
  Therefore the intersection with $P_{\lambda(C)}$ of the order ideal in $P_\lambda$ generated by $J(D)$ is contained in $I(C)$.
\end{proof}
\subsection{Decomposition of endomorphisms}
\label{sec:deco-endo}
Suppose that $A$ has the form (\ref{eq:11}).
Then define $A_C$ to be the subgroup
\begin{equation*}
  A_C=\{(a_1,\ldots,a_l)|a_k=0 \text{ if } k\notin S_C\}.
\end{equation*}
Thus $A_C$ is a finite abelian $p$-group of type $\lambda(C)$.
We have a decomposition
\begin{equation}
  \label{eq:18}
  A=\prod_{C\in \tilde \pi_0(J-I)} A_C.
\end{equation}

Denote the characteristic subgroups of $A_C$ corresponding to $I(C)$ and $J(C)$ (which are order ideals in $P_{\lambda(C)}$) by $A_{I,C}$ and $A_{J,C}$ respectively.
The decomposition (\ref{eq:18}) induces a decomposition
\begin{equation}
  \label{eq:19}
  A_J/A_I=\prod_{C\in \pi_0(J-I)} A_{J,C}/A_{I,C}.
\end{equation}
There is no contribution from $A_0$ since $A_{I,0}=A_{J,0}$.

With respect to the decomposition (\ref{eq:18}), every endomorphism of $A$ can be written as a square matrix $\{\phi_{CD}\}$, where $\phi_{CD}:A_D\to A_C$ is a homomorphism.
\begin{lemma}
  \label{lemma:sum-endo}
  Let $I\subset J$ be order ideals in $P_\lambda$.
  Then every endomorphism $\phi$ of $A$ induces an endomorphism
  \begin{equation*}
    \bar \phi: A_J/A_I\to A_J/A_I
  \end{equation*}
  such that $\bar\phi(A_{J,C}/A_{I,C})\subset A_{J,C}/A_{I,C}$ for each $C\in \pi_0(J-I)$, and
  \begin{equation*}
    \bar \phi= \bigoplus_{C\in \pi_0(J-I)} \overline{\phi_{CC}},
  \end{equation*}
  where $\overline{\phi_{CC}}$ is the endomorphism of $A_{J,C}/A_{I,C}$ induced by $\phi_{CC}$.
\end{lemma}
\begin{proof}
  By Theorem~\ref{theorem:degeneration} and Lemma~\ref{cor:ideal-components}, if $C\neq D$ then $\phi_{CD}(A_{J,D})\subset A_{I,C}$.
  Therefore, $\bar \phi$ remains unchanged if $\phi_{CD}$ is replaced by $0$ for all $C\neq D$.
  This amounts to replacing $\phi$ by $\oplus_C \phi_{CC}$ and the lemma follows.
\end{proof}
\subsection{Tensor product decomposition of invariant subspaces}
Let $I$ be a small order ideal.
We shall use the notation of Section~\ref{sec:deco-endo} with $J=I^\perp$.
For each $C\in \pi_0(I^\perp-I)$ let $K_C=A_C\times \hat A_C$, and let $\Sp(K_C)$ be the corresponding symplectic group.
Just as (by Theorem~\ref{theorem:characteristic-invariant}) $L^2(A_{I^\perp}/A_I)$ is an invariant subspace for the Weil representation of $\Sp(K)$ on $L^2(A)$, $L^2(A_{I^\perp,C}/A_{I,C})$ is an invariant subspace for the Weil representation of $\Sp(K_C)$ on $L^2(A_C)$.

Now, if $g\in \Sp(K)$, we may write $g=
\left(\begin{smallmatrix}
    g_{11}& g_{12}\\g_{21}&g_{22}
\end{smallmatrix}\right)
$ with respect to the decomposition $K=A\times \hat A$.
For convenience, we identify $\hat A$ with $A$ using $e_i\mapsto \epsilon_i$ for $i=1,\ldots,l$, where $e_i$ and $\epsilon_i$ are as in Section~\ref{sec:symplectic-orbits}.
Hence, we may think of each $g_{ij}$ as an endomorphism of $A$.
By Lemma~\ref{lemma:sum-endo}, the resulting endomorphism $\bar g_{ij}$ of $A_{I^\perp}/A_I$ preserves $A_{I^\perp,C}/A_{I,C}$ for each $C$.
It follows that the image of $\Sp(K)$ in $\Sp(K(I))$ (see Section~\ref{sec:invar-spac-corr}) is the product of the images of the $\Sp(K_C)$'s in the $\Sp(K_C(I\cap C))$'s as $C$ ranges over $\pi_0(I^\perp-I)$.

Thus, by Theorems~\ref{theorem:product-Weil} and~\ref{theorem:invar-I},
\begin{cor}
  \label{cor:component-decomposition}
  The Weil representation of $\Sp(K)$ on $L^2(A_{I^\perp}/A_I)$ is projectively equivalent to the tensor product of the Weil representations of $\Sp(K_C)$ on $L^2(A_{I^\perp,C}/A_{I,C})$ as $C$ ranges over $\pi_0(I^\perp-I)$.
\end{cor}
\section{Poset of invariant subspaces}
\label{sec:poset-invar-sub}
\subsection{The invariant subspaces}
Let
\begin{equation*}
  Q_\lambda=\big\{(I,\phi)|I\subset P_\lambda\text{ a small order ideal},\;\phi:\pi_0(I^\perp-I)\to \Z/2\Z \text{ any function}\big\}.
\end{equation*}
For each $(I,\phi)\in Q_\lambda$ use the decomposition of Corollary~\ref{cor:component-decomposition} to define $L^2(A)_{I,\phi}$ as the subspace of $L^2(A_{I^\perp}/A_I)$ given by
\begin{equation*}
  L^2(A)_{I,\phi}=\bigotimes_{C\in \pi_0(I^\perp-I)} L^2(A_{I^\perp,C}/A_{I,C})_{\phi(C)}
\end{equation*}
where $L^2(A_{I^\perp,C}/A_{I,C})_{\phi(C)}$ denotes the space of even or odd functions on $A_{I^\perp,C}/A_{I,C}$ when $\phi(C)$ is $0$ or $1$ respectively.
In other words, $L^2(A)_{I,\phi}$ consists of functions on $A_{I^\perp}/A_I$ which, under the decomposition
\begin{equation}
  \label{eq:20}
  A_{I^\perp}/A_I=\prod_{C\in \pi_0(I^\perp-I)} A_{I^\perp,C}/A_{I,C}
\end{equation}
are even in the components where $\phi(C)=0$ and odd in the components where $\phi(C)=1$.
By Theorems~\ref{theorem:even-odd} and~\ref{theorem:characteristic-invariant}, and by Corollary~\ref{cor:component-decomposition}, $L^2(A)_{I,\phi}$ is an $\Sp(K)$-invariant subspace of $L^2(A)$ for each $(I,\phi)\in Q_\lambda$.

\subsection{The partial order}
\label{sec:partial-order}
Clearly,
\begin{lemma}
  \label{lemma:poset}
  For $(I,\phi)$ and $(I',\phi')$ in $Q_\lambda$, $L^2(A)_{I',\phi'}\subset L^2(A)_{I,\phi}$ if and only if the following conditions are satisfied:
  \begin{enumerate}
  \item \label{item:3} $I\subset I'$.
  \item \label{item:4} For each $P\in \pi_0(I^\perp-I)$,
    \begin{equation*}
      \phi(P)=\sum_{P'\in \pi_0(I^{\prime\perp}-I'),P'\subset P}\phi'(P').
    \end{equation*}
  \end{enumerate}
\end{lemma}
Thus the conditions \ref{item:3} and \ref{item:4} define a partial order on $Q_\lambda$ (which is obviously independent of $p$).
Recall that the multiplicity $m(x)$ of an element $x=(v,k)\in P_\lambda$ is the number of times $k$ occurs in the partition $\lambda$.
For any subset $S\subset P_\lambda$, let $[S]$ denote the number of elements of $S$, counted with multiplicity:
\begin{equation*}
  [S]=\sum_{x\in S} m(x).
\end{equation*}
\begin{lemma}
  \label{lemma:dimension}
  For each $(I,\phi)\in Q_\lambda$,
  \begin{equation*}
    \dim L^2(A)_{I,\phi}=\prod_{C\in \pi_0(I^\perp-I)}\frac{p^{[C]}+(-1)^{\phi(C)}}2.
  \end{equation*}
\end{lemma}

\section{Irreducible subspaces}
\label{sec:parametrization}
\subsection{A bijection between $J(P_\lambda)$ and $Q_\lambda$}
\label{sec:biject-betw-p_lambda}
Let $J(P_\lambda)$ denote the lattice of order ideals in $P_\lambda$.
\begin{lemma}
  \label{lemma:cardinalities}
  For each partition $\lambda$, $|J(P_\lambda)|=|Q_\lambda|$. 
\end{lemma}
\begin{proof}
  We construct an explicit bijection $Q_\lambda\to J(P_\lambda)$.
  To $(I,\phi)\in Q_\lambda$, associate the ideal (see Lemma~\ref{lemma:ideals-components})
  \begin{equation*}
    \Theta(I,\phi)=I\bigcup\Bigg(\coprod_{C\in \pi_0(I^\perp-I)} I_{\phi(C)}\Bigg)
  \end{equation*}
  where
  \begin{equation*}
    I_{\phi(C)}=
    \begin{cases}
      I\cap C & \text{ if } \phi(C)=0\\
      I^\perp\cap C & \text{ if } \phi(C)=1.
    \end{cases}
  \end{equation*}
  In the other direction, given an ideal $J\subset P_\lambda$, $I=J\cap J^\perp$ is a small order ideal.
  We have $I^\perp=J\cup J^\perp$.
  For each $C\in \pi_0(I^\perp-I)$ define
  \begin{equation*}
    \phi_J(C)=
    \begin{cases}
      0 & \text{ if } I\cap C=J\cap C,\\
      1 & \text{ if } I\cap C=J^\perp\cap C.
    \end{cases}
  \end{equation*}
  Define $\Psi:Q_\lambda\to J(P_\lambda)$ by $\Psi(J)=(J\cap J^\perp,\phi_J)$ where $\phi_J$.
  It is easy to verify that $\Phi$ and $\Psi$ are mutual inverses.
\end{proof}
\subsection{Existence lemma}
\begin{lemma}
  \label{lemma:existence}
  For every $(I,\phi)\in Q_\lambda$, there exists $f\in L^2(A)_{I,\phi}$ such that $f\notin L^2(A)_{I',\phi'}$ for any $(I',\phi')<(I,\phi)$.
\end{lemma}
\begin{proof}
  Take as $f$ the unique element in $L^2(A)_{I,\phi}$ whose value at $a(I^\perp)+A_I$ (using the notation of Section~\ref{sec:orbits}) is $1$, and which vanishes on all elements of $A_{I^\perp}/A_I$ not obtained from $a(I^\perp)+A_I$ by changing the signs of some of its components under the decomposition (\ref{eq:20}).
\end{proof}
\subsection{The irreducible invariant subspaces}
\label{sec:irred-invar-subsp}
The two lemmas above are enough to give us the main theorem:
\begin{theorem}
  \label{theorem:main}
  For each $(I,\phi)\in Q_\lambda$, there is a unique irreducible subspace for the Weil representation of $\Sp(K)$ on $L^2(A)$ which is contained in $L^2(A)_{I,\phi}$ but not $L^2(A)_{I',\phi'}$ for any $(I',\phi')<(I,\phi)$.
  As $p$ varies, the dimension of this representation is a polynomial in $p$ of degree $[I^\perp-I]$ with leading coefficient $2^{-|\pi_0(I^\perp-I)|}$ and all coefficients in $\Z_{(2)}$.
\end{theorem}
\begin{proof}
  By Corollary~\ref{cor:Boolean-lattice}, the $\Sp(K)$-invariant subspaces of $L^2(A)$ form a Boolean lattice $\Lambda$.
  Let $R$ denote the set of minimal non-trivial $\Sp(K)$-invariant subspaces of $L^2(A)$.
  These are the atoms of $\Lambda$.
  By Corollary~\ref{cor:orbit-mult} and Theorem~\ref{theorem:symplectic-orbits} the cardinality of $R$ is the same as that of $J(P_\lambda)$.
  Each invariant subspace is determined by the atoms which are contained in it.
  The map $(I,\phi)\mapsto L^2(A)_{I,\phi}$ is an order-preserving map $Q_\lambda\to \Lambda$.
  Let $R_{I,\phi}$ be the set of atoms which occur in $L^2(A)_{I,\phi}$ but not in $L^2(A)_{I',\phi'}$ for any $(I',\phi')<(I,\phi)$.
  The subsets $R_{I,\phi}$ are $|Q_\lambda|$ pairwise disjoint subsets of $R$, and by Lemma~\ref{lemma:existence}, each of them is non-empty.
  Therefore, by Lemma~\ref{lemma:cardinalities}, each of them must be singleton and these subspaces exhaust $R$.
  It follows that there is a unique irreducible representation of $\Sp(K)$ that occurs in $L^2(A)_{I,\phi}$ but not in $L^2(A)_{I',\phi'}$ for any $(I',\phi')<(I,\phi)$.
  Let $V_{I,\phi}$ denote this irreducible subspace.

  By Lemma~\ref{lemma:dimension},
  \begin{equation*}
    \sum_{(I',\phi')\leq (I,\phi)}\dim V_{I',\phi'} = \prod_{P\in \pi_0(I^\perp-I)} \frac{p^{[C]}+(-1)^{\phi(C)}}2.
  \end{equation*}
  By the M\"obius inversion formula \cite[Section~3.7]{MR1442260},
  \begin{equation}
    \label{eq:10}
    \dim V_{I,\phi}=\sum_{(I',\phi')\leq (I,\phi)} \mu((I,\phi),(I',\phi'))\prod_{C\in \pi_0(I^{\prime\perp}-I')} \frac{p^{[C]}+(-1)^{\phi(C)}}2,
  \end{equation}
  where $\mu$ is a the M\"obius function of $Q_\lambda$.
  Since $\mu((I,\phi),(I,\phi))=1$ and the M\"obius function is integer-valued, the right hand side of (\ref{eq:10}) is indeed a polynomial in $p$ with leading coefficient $2^{-|\pi_0(I^\perp-I)|}$. Clearly, the other coefficients do not have denominators other than powers of $2$.
\end{proof}
\subsection{A combinatorial lemma}
\label{sec:combinatorial-lemma}
\begin{lemma}
  \label{lemma:combinatorial}
  Let $P$ be a poset and $J(P)$ be its lattice of order ideals.
  Let $m:P\to \mathbf N$ be any function (called the multiplicity function).
  For each $S\subset P$ let $[S]=\sum_{x\in S} m(x)$, the elements of $S$ counted with multiplicity, and $\max S$ denote the set of maximal elements of $S$.
  If $\alpha:J(P)\to \C[t]$ is a function such that
  \begin{equation*}
    \sum_{J\subset I}\alpha(J)=t^{[I]} \text{ for every order ideal } I\subset P,
  \end{equation*}
  then,
  \begin{equation}
    \label{eq:27}
    \alpha(I)=t^{[I]}\prod_{x\in \max I}(1-t^{-m(x)}).
  \end{equation}
\end{lemma}
\begin{proof}
  By the M\"obius inversion formula for a finite distributive lattice \cite[Example~3.9.6]{MR1442260},
  \begin{eqnarray}
    \nonumber
    \alpha(I)&=&\sum_{I-\max I\subset J\subset I}(-1)^{|I-J|}t^{[J]}\\
    \label{eq:14}
    &=&t^{[I]}\sum_{S\subset \max I}(-1)^{|\max I -S|}t^{-[\max I-S]}.
    \end{eqnarray}
    Each term in the expansion of the product
    \begin{equation*}
      \prod_{x\in \max I}(1-t^{-m(x)})
    \end{equation*}
    is obtained choosing a subset $S\subset \max I$ and taking
    \begin{equation*}
      \prod_{x\notin S} (-t^{-m(x)})=(-1)^{|\max I-S|}t^{-[\max I-S]}.
    \end{equation*}
    Therefore, the expression (\ref{eq:14}) for $\alpha(I)$ reduces to (\ref{eq:27}) as claimed.
  \end{proof}
\subsection{Explicit formula for the dimension}
\label{sec:formula-dimension}
Recall (from Section~\ref{sec:irred-invar-subsp}) that for each $(I,\phi)\in Q_\lambda$, $V_{I,\phi}$ denotes the unique irreducible $\Sp(K)$-invariant subspace of $L^2(A)$ which lies in $L^2(A)_{I,\phi}$ but not in any proper subspace of the form $L^2(A)_{I',\phi'}$.
We shall obtain a nice expression for $\dim V_{I,\phi}$ by applying Lemma~\ref{lemma:combinatorial} to the induced subposet of $P_\lambda$ given by
\begin{equation*}
  P_\lambda^+=\{(v,k)\in P_\lambda:v< (k-1)/2\}.
\end{equation*}
For each small order ideal $I\subset P_\lambda$, let $I^+=I^\perp\cap P_\lambda^+$.
Then $I\mapsto I^+$ is an order reversing isomorphism from the \poset{} of small order ideals in $P_\lambda$ to the \poset{} $J(P_\lambda^+)$ of all order ideals in $P_\lambda^+$.

  Let
  \begin{equation}
    \label{eq:21}
    V_I=\bigoplus_{\phi:\pi_0(I^\perp-I)\to \Z/2\Z} V_{I,\phi}.
  \end{equation}
  Denote by $V_I^0$ and $V_I^1$ the subspaces of even or odd functions in $V_I$ respectively.
\begin{lemma}
  \label{theorem:connected-dim}
  If $I\subset P_\lambda$ is a small order ideal, then for $\epsilon\in \{0,1\}$,
  \begin{equation*}
    \dim V_I^\epsilon=
    \begin{cases}
      (p^{[I^\perp-I]}+(-1)^\epsilon)/2 & \text{ if } I^+= \emptyset,\\
      p^{[I^\perp-I]}\prod_{x\in \max I^+}(1-p^{-2m(x)})/2 & \text{ otherwise.} 
    \end{cases}
  \end{equation*}
\end{lemma}
\begin{proof}
  Suppose $I\subset P_\lambda$ is a small order ideal.
  Then,
  \begin{equation}
    \label{eq:16}
    L^2(A_{I^\perp}/A_I)=\bigoplus_{J\supset I,\;J\text{ small}} V_J=\bigoplus_{J^+\subset I^+}V_J.
  \end{equation}
  Define $\alpha:J(P_\lambda^+)\to \C$ by $\alpha(J^+)=\dim V_J$.
  Comparing dimensions
  \begin{equation}
    \label{eq:15}
    \sum_{J^+\subset I^+} \alpha(J^+)=p^{[I^\perp-I]}.
  \end{equation}
  Let $E=\{(v,k)\in I^\perp-I|v=(k-1)/2\}$, the set of points in $I^\perp-I$ which lie on its axis of symmetry.
  Then $[I^\perp-I]=[E]+2[I^+]$.
  Therefore (\ref{eq:15}) becomes
  \begin{equation*}
    \sum_{J^+\subset I^+} \alpha(J^+)=p^{[E]}p^{2[I^+]}.
  \end{equation*}
  Taking $P=P_\lambda^+$ and setting $t=p^2$ in Lemma~\ref{lemma:combinatorial} gives
  \begin{eqnarray*}
    \dim V_I & = & p^{[E]+2[I^+]}\prod_{x\in \max I^+} (1-p^{-2m(x)})\\
    & = & p^{[I^\perp-I]}\prod_{x\in \max I^\perp}(1-p^{-2m(x)}).
  \end{eqnarray*}
  In order to obtain Lemma~\ref{theorem:connected-dim}, it remains to find the dimensions of the spaces of even and odd functions in $V_I$.
  If $I^+=\emptyset$ then $E=I^\perp-I$.
  In this case, $V_{I,\phi}$ is just the set of even or odd functions in $L^2(A_{I^\perp}/A_I)$ and has dimension as claimed.
  
  Otherwise, we proceed by induction on $I^+$.
  Thus assume that Lemma~\ref{theorem:connected-dim} holds for small order ideals $I'\supsetneq I$.
  The space of even functions in $L^2(A_{I^\perp}/A_I)$ has dimension one more than the space of odd functions.
  Breaking up the spaces in (\ref{eq:16}) into even and odd functions, we see this difference is accounted for by the summand corresponding to $J^+=\emptyset$, as discussed above. By the induction hypothesis, the dimensions of even and odd parts of the summands corresponding to $\emptyset\subsetneq J^+\subsetneq I^+$ are equal.
  Therefore, the even and odd parts of $V_I$ must have the same dimension.
\end{proof}
\begin{theorem}
  \label{theorem:dim}
  If $I\subset P_\lambda$ is a small order ideal, then
  \begin{equation*}
    \dim V_{I,\phi}=\prod_{C\in \pi_0(I^\perp-I)} \dim V_{I(C),\phi(C)},
  \end{equation*}
  where, since $I(C)^\perp-I(C)$ is connected, $\dim V_{I(C),\phi(C)}$ is given by Lemma~\ref{theorem:connected-dim}.
\end{theorem}
\subsection{Examples}
\label{sec:examples}
We begin with the case $A=(\Z/p^k\Z)^l$, corresponding to $\lambda=(k,\ldots,k)$ (repeated $l$ times).
$P_\lambda$ is then a linear order, with $k$ points.
$Q_\lambda$ has two linear components, consisting of the even and odd parts.
An informative way to display the decomposition of $L^2(A)$ is as the Hasse diagram of $Q_\lambda$, but with the dimension of the corresponding irreducible invariant subspace in place of each vertex.
In this case we get
\begin{equation*}
  \xymatrix{
    \frac{p^{lk}(1-p^{-2})}2 \ar@{-}[d] & \frac{p^{lk}(1-p^{-2})}2 \ar@{-}[d]\\
    \frac{p^{l(k-2)}(1-p^{-2})}2 \ar@{-}& \frac{p^{l(k-2)}(1-p^{-2})}2 \ar@{-}\\
    \vdots & \vdots\\
    \frac{p^{(k-2(\lfloor k/2\rfloor-1))}(1-p^{-2})}2\ar@{-}[d] & \frac{p^{(k-2(\lfloor k/2\rfloor-1))}(1-p^{-2})}2\ar@{-}[d]\\
    \frac{p^{(k-2\lfloor k/2 \rfloor)}+1}2 & \frac{p^{(k-2\lfloor k/2 \rfloor)}-1}2
  }
\end{equation*}
The entry at the bottom right is zero when $k$ is even and should be omitted.
This is consistent with the previously known results of Prasad \cite{MR1478492} and Cliff-McNeilly-Szechtman \cite{MR1783635}.
The picture for $\lambda=(2,1)$ is same as that for $\lambda=(3)$.
Perhaps the simplest non-trivial example is $\lambda=(3,1)$ (it is the smallest example where $J(P_\lambda)$ is not a chain). We get
\begin{equation*}
  \begin{array}{cc}
    \xymatrix{
      & \frac{p^4-p^2}2 \ar@{-}[dl] \ar@{-}[d]\\
      \frac{(p+1)^2}4 & \frac{(p-1)^2}4
    }&
    \xymatrix{
      \frac{p^4-p^2}2 \ar@{-}[d] \ar@{-}[dr]&\\
      \frac{p^2-1}4 & \frac{p^2-1}4
    }
  \end{array}
\end{equation*}
For $\lambda=(3,2,1)$, we get
\begin{equation*}
  \begin{array}{cc}
    \xymatrix{
      & \frac{p^6-p^4}2\ar@{-}[d]\\
      & \frac{p^4-p^2}2\ar@{-}[dl]\ar@{-}[d]\\
      \frac{(p+1)^2}4&\frac{(p-1)^2}4
    }
    &
    \xymatrix{
      \frac{p^6-p^4}2\ar@{-}[d]&\\
      \frac{p^4-p^2}2\ar@{-}[d]\ar@{-}[dr]&\\
      \frac{p^2-1}4&\frac{p^2-1}4
    }
  \end{array}
\end{equation*}
For $\lambda=(4,2)$, we have
\begin{equation*}
  \begin{array}{cc}
    \xymatrix{
      & \frac{p^6-p^4}2\ar@{-}[d]\\
      & \frac{p^4-2p^2+1}2\ar@{-}[dl]\ar@{-}[d]\\
      \frac{p^2-1}2\ar@{-}[dr]&\frac{p^2-1}2\ar@{-}[d]\\
      & 1
    }
    &
    \xymatrix{
      \frac{p^6-p^4}2\ar@{-}[d]&\\
      \frac{p^4-2p^2+1}2\ar@{-}[d]\ar@{-}[dr]&\\
      \frac{p^2-1}2&\frac{p^2-1}2\\
      &
    }
  \end{array}
\end{equation*}
For $\lambda=(4,3,2,1)$, we have
{\small
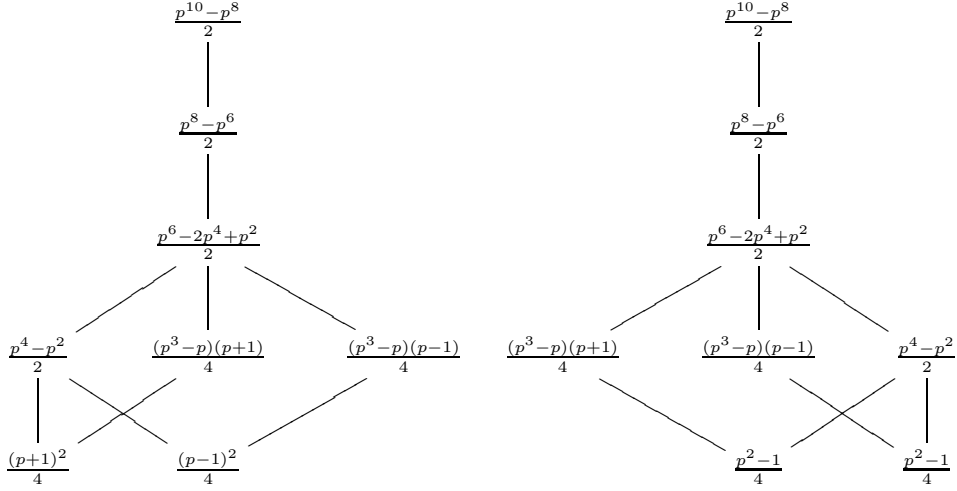
\begin{equation*}
  \begin{array}{cc}
    \xymatrix{
      & \frac{p^{10}-p^8}2\ar@{-}[d] &\\
      & \frac{p^8-p^6}2\ar@{-}[d] &\\
      & \frac{p^6-2p^4+p^2}2\ar@{-}[d]\ar@{-}[dl]\ar@{-}[dr] &\\
      \frac{p^4-p^2}2\ar@{-}[d]\ar@{-}[dr]&\frac{(p^3-p)(p+1)}4\ar@{-}[dl]
      &\frac{(p^3-p)(p-1)}4\ar@{-}[dl]\\
      \frac{(p+1)^2}4 & \frac{(p-1)^2}4&
    }
    &
    \xymatrix{
      & \frac{p^{10}-p^8}2\ar@{-}[d] &\\
      & \frac{p^8-p^6}2\ar@{-}[d] &\\
      & \frac{p^6-2p^4+p^2}2\ar@{-}[d]\ar@{-}[dl]\ar@{-}[dr] &\\
      \frac{(p^3-p)(p+1)}4\ar@{-}[dr]&\frac{(p^3-p)(p-1)}4
      \ar@{-}[dr]&\frac{p^4-p^2}2\ar@{-}[dl]\ar@{-}[d]\\
      &\frac{p^2-1}4 & \frac{p^2-1}4
    }
  \end{array}
\end{equation*}
}
\subsection{Projections onto the irreducible subspaces}
\label{sec:proj-onto-irred}
For each $(I,\phi)\in Q_\lambda$, let $E_{I,\phi}$ denote the projection operator onto $V_{I,\phi}$.
Recall from Lemma~\ref{lemma:onb} that the set of Weyl operators
\begin{equation*}
  \{W_k:k\in K\}
\end{equation*}
is an orthonormal basis of $\End_\C L^2(A)$.
Therefore, we may write
\begin{equation*}
  E_{I,\phi}=\sum_{k\in K} e_k(I,\phi)W_k
\end{equation*}
for some scalars $e_k(I,\phi)$.
The goal of this section is to show that this expansion is completely combinatorial.
More precisely, by Theorem~\ref{theorem:symplectic-orbits}, each $\Sp(K)$-orbit in $K$ corresponds to an order ideal in $P_\lambda$.
We shall show that if $k$ lies in the $\Sp(K)$-orbit corresponding to the order ideal $J$, then $e_k(I,\phi)$ is a polynomial in $p$ whose coefficients depend only on the combinatorial data $I$, $\phi$, and $J$.

In Section~\ref{sec:relat-comm} we saw that
\begin{equation*}
  \{\Delta_L:L\in J(P_\lambda)\}
\end{equation*}
is a basis of $\End_{\Sp(K)}L^2(A)$.
Therefore, we may write
\begin{equation*}
  E_{I,\phi}=\sum_{L\subset P_\lambda} \alpha_L(I,\phi) \Delta_L,
\end{equation*}
for some constants $\alpha_L(I,\phi)$.
If $k$ lies in the orbit corresponding to $J$ then
\begin{equation*}
  e_k(I,\phi)=\sum_{L\supset J}\alpha_L(I,\phi).
\end{equation*}
Therefore, it suffices to show that the $\alpha_L(I,\phi)$ are polynomials in $p$ whose coefficients are determined by the combinatorial data $L$, $I$, and $\phi$ (Theorem~\ref{theorem:alpha-qualitative}).
In fact, Theorems~\ref{theorem:alpha-supp} and~\ref{theorem:alpha-exact} compute $\alpha_L(I,\phi)$ explicitly.

To begin with, consider the case where $I^\perp-I$ is connected.
If $E_I$ is the projection operator onto $V_I$ (defined by (\ref{eq:21})), then by \ref{item:1},
\begin{equation*}
  |A|^{-1}p^{[I^\perp-I]}\Delta_I = \sum_{J^+\subset I^+} E_J.
\end{equation*}
Using M\"obius inversion for a finite distributive lattice as in Section~\ref{sec:combinatorial-lemma},
\begin{equation*}
  |A|E_I=\sum_{I^+-\max I^+\subset J^+\subset I^+} (-1)^{|I^+-J^+|}p^{[J^\perp-J]}\Delta_J.
\end{equation*}
Since $V_{I,\phi}$ consists of even or odd functions in $V_I$ (depending on whether $\phi(I^\perp-I)$ is $0$ or $1$), by (\ref{eq:22}), $E_{I,\phi}$ is given by
\begin{equation*}
  E_{I,\phi}=E_I(\mathrm{id}_{L^2(A)}+(-1)^{\phi(I^\perp-I)} |A|^{-1}\Delta_{P_\lambda})/2.
\end{equation*}
By Lemma~\ref{lemma:product},
\begin{equation*}
  (|A|^{-1}p^{[J^\perp-J]}\Delta_J)(|A|^{-1}\Delta_{P_\lambda})=|A|^{-1}\Delta_{J^\perp}.
\end{equation*}
Therefore when $I^\perp-I$ is connected 
\begin{equation}
  \label{eq:23}
  2|A|E_{I,\phi}=\sum_{I^+-\max I^+\subset J^+\subset I^+}(-1)^{|I^+-J^+|}(p^{[J^\perp-J]}\Delta_J+(-1)^{\phi(I^\perp-I)}\Delta_{J^\perp}).
\end{equation}

Now take $I\subset P_\lambda$ to be any small order ideal.
The decomposition (\ref{eq:18}) gives
\begin{equation*}
  L^2(A)=\bigotimes_{C\in \tilde\pi_0(I^\perp-I)} L^2(A_C)
\end{equation*}
and
\begin{equation*}
  V_{I,\phi}=\Big(\bigotimes_{C\in \pi_0(I^\perp-I)} V_{I(C),\phi(C)}\Big)\otimes L^2(A_{I^\perp(0)}/A_{I(0)}),
\end{equation*}
the last factor being one dimensional (since $I(0)=I^\perp(0)$).
So we have
\begin{equation}
  \label{eq:25}
  E_{I,\phi}=\Big(\bigotimes_{C\in \pi_0(I^\perp-I)} E_{I(C),\phi(C)}\Big)\otimes \Delta_{I(0)},
\end{equation}
where, since $I(C)^\perp-I(C)$ is connected, $E_{I(C),\phi(C)}$ is determined by (\ref{eq:23}).
A typical term in the expansion (\ref{eq:25}) will be of the form
\begin{equation}
  \label{eq:26}
  \Big(\bigotimes_{C\in \pi_0(I^\perp-I)}\Delta_{L(C)}\Big)\otimes \Delta_{I(0)},
\end{equation}
where, for each $C\in \pi_0(I^\perp-I)$, $I(C)\subset L(C)\subset I^\perp(C)$ with either $L(C)$ or $L(C)^\perp$ is a small order ideal in $P_{\lambda(C)}$.
But this is just $\Delta_L$, where
\begin{equation}
  \label{eq:24}
  L=I\bigcup \Big(\coprod_{C\in \pi_0(I^\perp-I)}L(C)\Big),
\end{equation}
is an order ideal in $P_\lambda$ by Lemma~\ref{lemma:ideals-components}.
We have the qualitative result
\begin{theorem}
  \label{theorem:alpha-qualitative}
  For each $(I,\phi)\in Q_\lambda$, $2^{|\pi_0(I^\perp-I)|}|A|\alpha_L(I,\phi)$ is a polynomial in $p$ whose coefficients are integers which depend only on the combinatorial data $I$, $\phi$ and $L$.
\end{theorem}
Let $I_L=L\cap L^\perp$.
Examining (\ref{eq:23}) more carefully gives
\begin{theorem}
  \label{theorem:alpha-supp}
  The coefficient $\alpha_L(I,\phi)$ is non-zero if and only if the following conditions hold:
  \begin{enumerate}
  \item \label{item:5} For each $C\in \pi_0(I^\perp-I)$, either $L(C)$ or $L(C)^\perp$ is a small order ideal in $P_{\lambda(C)}$.
  \item \label{item:6} $I^+-\max I^+\subset I_L^+\subset I^+$.
  \end{enumerate}
\end{theorem}
\begin{proof}
  For $\alpha_L(I,\phi)$ to be non-zero, it is necessary that $L$ be of the form (\ref{eq:24}) for some order ideals $L(C)$ of $P_{\lambda(C)}$ which occur in the right hand side of (\ref{eq:23}).
  Furthermore, since each order ideal in $P_{\lambda(C)}$ appears at most once in the right hand side of (\ref{eq:23}), so each order ideal in $P_\lambda$ appears only once in the expansion (\ref{eq:25}).
  In particular, no cancellation is possible, and for all such ideals $\alpha_L(I,\phi)\neq 0$.

  Now $L(C)$ appears on the right hand side of (\ref{eq:23}) if and only if \ref{item:5} holds, and $I(C)^+-\max I(C)^+\subset I_L(C)^+\subset I(C)^+$.
Since $\max I^+=\coprod_C \max I(C)^+$, this amounts to the condition \ref{item:6}.
\end{proof}
If these conditions do hold, then for each $C'\in \pi_0(I_L^\perp-I_L)$ there exists $C\in \pi_0(I^\perp-I)$ such that $C'\subset C$.
Furthermore, $\phi_L(C')$ depends only on $C$, so we may denote its value by $\phi_L(C)$.
For $I=I(C)$, the right hand side of (\ref{eq:23}) can be written as
\begin{equation*}
  \sum_{L(C)}(-1)^{|I(C)^+-L(C)^+|+\phi(C)\phi_L(C)}p^{[I_L(C)^\perp-L(C)]},
\end{equation*}
the sum being over an appropriate set of order ideals $L(C)\subset P_{\lambda(C)}$.
Write $\langle \phi_1,\phi_2\rangle$ for $\sum_{C\in \pi_0(I^\perp-I)} \phi_1(C)\phi_2(C)$ for any functions $\phi_i:\pi_0(I^\perp-I)\to \Z/2\Z$.
The additive nature of the exponents in the above expression allows us to get an exact expression of $\alpha_L(I,\phi)$:
\begin{theorem}
  \label{theorem:alpha-exact}
  If an order ideal $L\subset P_\lambda$ satisfies the conditions of Theorem~\ref{theorem:alpha-supp}, then
  \begin{equation*}
    2^{|\pi_0(I^\perp-I)|}|A|\alpha_L(I,\phi)=(-1)^{|I^+-I_L^+|+\langle\phi,\phi_L\rangle}p^{[I_L^\perp-L]}.
  \end{equation*}
\end{theorem}
\section{Finite modules over a Dedekind domain}
\label{sec:finite-modules}
Let $F$ be a non-Archime\-dean local field with ring of integers $R$.
Let $P$ denote the maximal ideal of $R$.
Assume that the residue field $R/P$ is of odd order $q$.
Fix a continuous character $\psi:F\to U(1)$ whose restriction to $R$ is trivial, but whose restriction to $P^{-1}R$ is not (see, for example, Tate's thesis \cite{MR0217026}).
Then if $\psi_x(y)=\psi(xy)$, the map $x\mapsto \psi_x$ is an isomorphism of $F$ into $\hat F$.
Under this isomorphism, $R$ has image $R^\perp=\widehat{F/R}$.
More generally, $P^{-n}$ has image $(P^n)^\perp=\widehat{F/P^n}$ for every integer $n$ (recall that for positive $n$, $P^{-n}$ is the set of elements $x\in F$ such that $x P^n\in R$).
Thus, it gives rise to an isomorphism $P^{-n}/R\to \widehat{R/P^n}$ for each positive integer $n$.
Since $P^{-n}/R$ inherits the structure of an $R$-module, this isomorphism also allows us to think of $\widehat{R/P^n}$ as an $R$-module.
Now suppose $A$ is a finitely generated torsion module over $R$.
Then
\begin{equation}
  \label{eq:13}
  A=R/P^{\lambda_1}\times\cdots\times R/P^{\lambda_l}
\end{equation}
for a unique partition $\lambda$.
By the discussion above, $\hat A$ is also an $R$-module (non-canonically isomorphic to $A$).
Let $K=A\times \hat A$, and $\Sp(K)$ be as in Section~\ref{theorem:SvN}.
Define $\Sp_R(K)$ to be the subgroup of $\Sp(K)$ consisting of $R$-module automorphisms.

The Weil representation of $\Sp_R(K)$ is simply the restriction of the Weil representation of $\Sp(K)$ on $L^2(A)$ to $\Sp_R(K)$.
All the theorems and proofs in this article concerning finite abelian $p$-groups generalize to the Weil representation of $\Sp_R(K)$ on $L^2(A)$, so long as $p$ is replaced by $q$ in the formulas.
Since every finitely generated torsion module over a Dedekind domain is a product of its primary components, and module automorphisms respect the primary decomposition, the reduction in Section~\ref{sec:prim-decomp} works for finite modules of odd order over Dedekind domains.

Singla \cite{Poojareps,1101.3696} has proved that the representation theory of $G(R/P^2)$, where $G$ is a classical group, depends on $R$ only through $q$, the order of the residue field.
More precisely, if $R$ and $R'$ are two discrete valuation rings and an isomorphism between their residue fields is fixed (for example, take $R=\Z_p$, the ring of $p$-adic integers, and $R'=(\Z/p\Z)[[t]]$, the ring of Laurent series with coefficients in $\Z/p\Z$), then there is a canonical bijection between the irreducible representations of $G(R/P^2)$ and $G(R'/P^2)$ which preserves dimensions.
There is also a canonical bijection between their conjugacy classes which preserves sizes.
All existing evidence points towards the existence of a similar correspondence automorphism groups of modules of type $\lambda$ (see, for example \cite[Conjecture~1.2]{MR2456275}).
The results in this paper also point in the same direction: for each partition $\lambda$, there is a canonical correspondence between the invariant subspaces of the Weil representations associated to the finitely generated torsion $R$-module of type $\lambda$ and the finitely generated torsion $R'$-module of type $\lambda$ which preserves dimensions.
\subsection*{Acknowledgment}
We thank Uri Onn for some helpful comments on a preliminary version of this paper. We thank the referee for some suggestions to improve the readability of this paper.
Kunal Dutta was a student at the Institute of Mathematical Sciences, Chennai when this paper was written.
He thanks the Institute for its hospitality and support.
\def\cprime{$'$}

\end{document}